\documentclass[12pt]{article}
\usepackage{amsfonts,amsmath,amssymb}
\usepackage{epsfig}
\usepackage{latexsym}
\usepackage{color}
\usepackage{tikz}
\usepackage{amsthm}
\usepackage{hyperref}
\usepackage{ifthen}

\newtheorem{theorem}{Theorem}
\newtheorem{lemma}[theorem]{Lemma}

\newtheorem{cor}[theorem]{Corollary}
\newtheorem{proposition}[theorem]{Proposition}

\newtheorem{problem}{Problem}

\def\vertex(#1){\put(#1){\circle*{2}}}
\def\vertexo(#1){\put(#1){\circle{2}}}
\def\vert(#1){\put(#1){\circle*{1.5}}}
\def\verto(#1){\put(#1){\circle{1.5}}}
\def\lab(#1)#2{\put(#1){\makebox(0,0)[c]{#2}}}
\newcommand{\edim}{{\rm edim}}

\title{A note on the metric and edge metric dimensions of 2-connected graphs
\footnotetext{\small\tt knor@math.sk,
skrekovski@gmail.com, ismael.gonzalez@uca.es}}

\author{Martin Knor$^1$, Riste \v Skrekovski$^{2,3}$ and Ismael G. Yero$^4$\\[0.3cm]
\small $^1$ \it Slovak University of Technology in Bratislava, Bratislava, Slovakia \\[0.1cm]
\small $^2$ \it University of Ljubljana, FMF, 1000 Ljubljana, Slovenia \\[0.1cm]
\small $^3$ \it Faculty of Information Studies, 8000 Novo Mesto, Slovenia \\[0.1cm]
\small $^4$ \it Universidad de C\'adiz, Departamento de Matem\'aticas, Algeciras, Spain\\
}

\begin{document}
\maketitle

\begin{abstract}
For a given graph $G$, the metric and edge metric dimensions of $G$,
$\dim(G)$ and $\edim(G)$, are the cardinalities of the smallest possible
subsets of vertices in $V(G)$ such that they uniquely identify
the vertices and the edges of $G$, respectively, by means of distances.
It is already known that metric and edge metric dimensions are not
in general comparable.
Infinite families of graphs with pendant vertices in which
the edge metric dimension is smaller than the metric dimension are
already known.
In this article, we construct a 2-connected graph $G$ such that
$\dim(G)=a$ and $\edim(G)=b$ for every pair of integers $a,b$,
where $4\le b<a$.
For this we use subdivisions of complete graphs, whose metric
dimension is in some cases smaller than the edge metric dimension.
Along the way, we present an upper bound for the metric and edge
metric dimensions of subdivision graphs under some special conditions.
\end{abstract}

{\it Keywords:} Edge metric dimension; metric dimension; 2-connected
graphs; edge subdivision.

{\it AMS Subject Classification numbers:}   05C12; 05C76





\section{Introduction}

The definition of edge metric dimension of graphs appeared first
in \cite{Kel} motivated by the problem of uniquely locating
the connection between two vertices of a graph, that is the edge;
and in addition, to giving more insight into the classical metric
dimension.

Given a connected graph $G$, a vertex $v\in V(G)$ and an edge $e=xy\in E(G)$,
the \emph{distance} between $e$ and $v$ is $d_G(e,v)=\min\{d_G(x,v),d_G(y,v)\}$,
where $d_G(x,v)$ stands for the standard vertex distance in graphs,
\emph{i.e.}, the length of a shortest $x$-$v$ path.
A set of vertices $S\subset V(G)$ is an \emph{edge metric generator} for $G$
if for any pair of distinct edges $e,f\in E(G)$, there exists a vertex
$v\in S$ such that $d_G(e,v)\ne d_G(f,v)$.
The cardinality of the smallest possible edge metric generator for $G$ is
the \emph{edge metric dimension} of $G$, denoted by $\edim(G)$.
An edge metric generator of cardinality $\edim(G)$ is
an \emph{edge metric basis} of $G$.
Any edge metric generator uniquely identifies (resolves or recognizes)
all the edges of the graph. Concepts of metric generators, metric basis
and metric dimension have analogous definitions, if we want to uniquely
recognize the vertices of a graph instead of the edges.
These three latter concepts were firstly and independently defined in \cite{Harary1976}
and \cite{Slater1975}.
The metric dimension of a graph $G$ is denoted by $\dim(G)$.
Recent results concerning edge metric dimension of graphs can be found
in \cite{Filipovi2019,Geneson2020,Peterin2020,Zhu,Zubrilina}.

Relationships between the metric dimension and the edge metric dimension
of graphs have became attracting for some researchers from the moment
in which the seminal article \cite{Kel} on edge metric dimension appeared.
One reason for this attraction is based on the not clear comparability
of such parameters.
That is, the classical metric dimension can be smaller, equal or larger
than the edge metric dimension as first proved in \cite{Kel}.
Several different infinite families of graphs having metric dimension
smaller or equal than the edge metric dimension were already known
from \cite{Kel}, and in contrast, only one infinite family
(the torus graphs $C_{4r}\square C_{4t}$) with metric dimension larger
than the edge metric dimension was known till recently.
Moreover, for this family the difference between the metric and edge metric
dimensions is just $1$.
Thus, a natural question raised up in \cite{Kel} was regarding finding
some other families of graphs $G$ satisfying that $\edim(G)<\dim(G)$.
From \cite{Zubrilina}, it was also known that the edge metric dimension
cannot be bounded from above by a factor of the metric dimension,
and thus, a similar question was there stated for the opposite direction.

The problem of finding some other families of graphs having edge metric
dimension smaller than the metric dimension remained open till recently,
when it was shown in \cite{Knor2020}, that for every $a$ and $b$,
where $2\le b<a$, there is a graph $G$ such that $\dim(G)=a$ and $\edim(G)=b$.
Hence, the metric dimension cannot be bounded from above by a factor of
the edge metric dimension as well, giving answer to problems described
in \cite{Kel} and \cite{Zubrilina}.
Notwithstanding, the graphs used in \cite{Knor2020} have pendant vertices,
which is not happening for the already known torus graphs from \cite{Kel}.
Thus, one may consider the question of finding families of graphs $G$
for which $\edim(G)<\dim(G)$, and having no vertices of degree one. Even more general, one would wonder if there is a relationship between constructing such kind of families of graphs regarding the $k$-connectedness, for $k\ge 2$, of the graphs of such families.
In this work we center our attention in the case of $2$-connected graphs
(graphs without cut vertices), also known as \emph{block graphs},
and we prove the following statement, which is the main result of this work.

\begin{theorem}
\label{thm:main}
Let $a$ and $b$ be integers, such that $4\le b<a$.
Then, there exists a $2$-connected graph $G$ such that $\dim(G)=a$ and
$\edim(G)=b$.
\end{theorem}

It remains unsolved if there is an analogue of Theorem~{\ref{thm:main}} when $b=2$ or $b=3$.
So, we point out the following open problem. Note that a very particular case $b=2$ and $a=3$
follows by Lemma~\ref{lem:difference1}.

\begin{problem}
Are there graphs $G$ with $\edim(G)=b$ and $\dim(G)=a$ for each $b\in\{2,3\}$ and $a>b$?
\end{problem}

Observe that the torus graphs are $3$-connected. In this sense, if there exist $3$-connected graphs $G$ with $\edim(G)=6$ and
$\dim(G)=7$, in which some edge metric basis is a subset of a metric basis, then by using the technique presented in this paper, and in \cite{Knor2020}, there may be a chance to construct other $3$-connected graphs satisfying similar conclusions as in
Theorem~{\ref{thm:main}} (probably with a little bit worse lower bound on $b$), for this particular case.
Unfortunately, at the moment the only known $3$-connected graphs whose edge
metric dimension is smaller than the metric dimension are the above
mentioned torus graphs, for which $\dim(G)=4$ and $\edim(G)=3$.
In consequence, we next state the following problem.

\begin{problem}
Let $b\ge6$. Are there $3$-connected graphs $G_b$ such that $\dim(G_b)=b+1$ and
$\edim(G_b)=b$ which have an edge metric basis which is a subset of some
metric basis?
\end{problem}

And more generally:

\begin{problem}
For which $k$ there exist $k$-connected graphs $G$ with $\dim(G)>\edim(G)$?
\end{problem}

In the next section, we present several results that are then used to prove Theorem \ref{thm:main}.


\section{Metric and edge dimensions of subdivision graphs}

The smallest value of $n$, $n\ge 3$, for which there exist graphs of order $n$
with edge metric dimension smaller than the metric dimension,
is $n=10$, see \cite{Knor2020}.
From \cite[Figure 1]{Knor2020}, we observe that there is just one 2-connected
graph with such order, which is indeed a subdivision of $K_4$.
In this sense, we consider the class of subdivision graphs, since it seems
that they could be good candidates for graphs which have edge metric dimension
strictly smaller than the metric dimension.

Given a graph $G$, by \emph{subdividing} one edge $uv$, we mean removing
the edge $uv$, and adding an extra vertex $w$ and the edges $uw$ and $wv$.
A \emph{subdivision graph} is obtained from a graph $G$ by subdividing all
its edges, and is denoted by $S(G)$.
Note that, if $G$ has $n$ vertices and $m$ edges, then $S(G)$ is a bipartite
graph with $n+m$ vertices and $2m$ edges.
In order to prove Theorem~{\ref{thm:main}}, we first need the following results,
which are bounds for the metric and edge metric dimensions of subdivision graphs,
under some restrictions for the original not subdivided graphs.
The bounds themselves are of interest based on the fact that tight useful bounds
for the metric and edge metric dimensions of graphs are not commonly existing.

\begin{theorem}
\label{thm:bound}
Let $G$ be a graph on $n$ vertices.
If $G$ contains $\left\lfloor\frac n3\right\rfloor$ vertex-disjoint paths of length
$2$, then $\dim(S(G))\le\left\lceil\frac{2n}{3}\right\rceil$.
\end{theorem}

\begin{proof}
We construct a set $T$ of the required cardinality and show that
such $T$ is a metric generator for $S(G)$.
Denote the vertices of $G$ by $v_1,v_2,\dots,v_n$, so that the $\left\lfloor\frac n3\right\rfloor$ vertex-disjoint paths
of length $2$ are $v_1v_2v_3$, $v_4v_5v_6$,$\dots$,
$v_{\alpha-2}v_{\alpha-1}v_{\alpha}$, where $\alpha=3\lfloor\frac n3\rfloor$.
Further, if $v_iv_j\in E(G)$, then denote by $x_{i,j}$ the vertex of degree
$2$ connected to $v_i$ and $v_j$ in $S(G)$.
Let $T'=\{x_{1,2},x_{2,3},x_{4,5},x_{5,6},\dots,x_{\alpha-2,\alpha-1},x_{\alpha-1,\alpha}\}$.
Now, let $T=T'$ if $n\equiv 0\pmod 3$, $T=T'\cup\{v_n\}$ if $n\equiv 1\pmod 3$,
and $T=T'\cup\{v_{n-1},v_n\}$ if $n\equiv 2\pmod 3$.
Obviously, $|T|=\lceil\frac{2n}3\rceil$.
So, it remains to show that $T$ is a metric generator for $S(G)$.

Consider the vertex $x_{1,2}$ of $T$.
This vertex partitions the set of vertices of $S(G)$ into several sets
according to their distance from $x_{1,2}$.
Clearly, $x_{1,2}$ distinguishes any two vertices belonging
to two different such sets.
Thus, it suffices to examine pairs of vertices, say $y$ and $z$,
belonging to a common set.
Since there is only one vertex at distance $0$ from $x_{1,2}$, we consider
the sets of vertices at distance $1$ and $2$ from $x_{1,2}$.
As will be shown below, it is not necessary to consider sets of vertices
at larger distance from $x_{1,2}$.

\bigskip
\noindent
{\bf Case 1}: $d(x_{1,2},y)=d(x_{1,2},z)=1$.
Then $y=v_1$ and $z=v_2$.
We have $d(x_{2,3},v_1)=3$ and $d(x_{2,3},v_2)=1$, and so $y,z$ are identified by $x_{2,3}$.

\bigskip
\noindent
{\bf Case 2}: $d(x_{1,2},y)=d(x_{1,2},z)=2$.
We may assume that $y,z\ne x_{2,3}$, otherwise $x_{2,3}$ distinguishes
$y$ and $z$.
If $y=x_{1,j}$ and $z=x_{2,k}$, then $d(x_{2,3},z)=2$ while $d(x_{2,3},y)\ge 2$
and the equality is attained ony if $j=3$.
But, in that case $v_k$ is at distance at most $1$ from a vertex of $T$ and
this vertex distinguishes $y$ and $z$.
So, we may consider the case when $y=x_{2,j}$ and $z=x_{2,k}$, where $j<k$.
The case when $y=x_{1,j}$ and $z=x_{1,k}$ can be solved analogously.
We also assume that $j,k\ne 3$, otherwise $x_{2,3}$ distinguishes $y$ and
$z$. If $v_k\in T$, then this vertex distinguishes $y$ and $z$.
Thus, it remains to consider the case when $v_j$ and $v_k$ are at distance
$1$ from vertices of $T'$.
But then there is such a vertex of $T'$ which is at distance
$1$ from one of $v_j,v_k$ and at distance at least $3$ from the other.
And this vertex distinguishes $y$ and $z$. This establishes the case.

\bigskip

Using analogous arguments for the remaining vertices of $T'$, we see that the only
vertices of $S(G)$ which may not be distinguished by $T$ are at distance
at least $3$ from every vertex of $T'$.
However, there is not such a vertex if $n\equiv 0\pmod 3$.
If $n\equiv 1\pmod 3$, then there is just one such vertex, namely $v_n$.
Finally, if $n\equiv 2\pmod 3$ it remains to consider three vertices
(or two if $v_{n-1}v_n\notin E(G)$), namely $v_{n-1}$, $x_{n-1,n}$ and $v_n$.
But these vertices are distinguished by either of $v_{n-1}$ and $v_n$.
And since $v_{n-1},v_n\in T$, the set $T$ is a metric generator for $S(G)$ also
in this case. This concludes the proof.
\end{proof}

We next present an analogous result of the previous one, for the edge metric dimension.

\begin{theorem}
\label{thm:ebound}
Let $G$ be a graph on $n$ vertices.
If $G$ contains $\lfloor\frac{n-1}3\rfloor$ vertex-disjoint paths of length
$2$, then $\edim(S(G))\le\lceil\frac{2n-2}3\rceil$.
\end{theorem}

\begin{proof}
Analogously as in the proof of Theorem~{\ref{thm:bound}}, we construct
a set $T$ of the required size, and show that $T$ is an edge metric generator for $S(G)$.
As in Theorem~{\ref{thm:bound}}, we use a similar notation for the vertices of $G$, $S(G)$, and the set $T'$,
but now considering $\alpha=3\lfloor\frac{n-1}3\rfloor$,
that is $T'=\{x_{1,2},x_{2,3},x_{4,5},x_{5,6},\dots,x_{\alpha-2,\alpha-1},x_{\alpha-1,\alpha}\}$.
Now, let $T=T'$ if $n\equiv 1\pmod 3$; $T=T'\cup\{v_{n-1}\}$ if $n\equiv 2\pmod 3$;
and $T=T'\cup\{v_{n-2},v_{n-1}\}$ if $n\equiv 0\pmod 3$.
Then $v_n$ is not in $T$ and it is not adjacent to a vertex of $T$.
Obviously, $|T|=\lceil\frac{2n-2}3\rceil$.
It remains to show that $T$ is a metric generator for $S(G)$.

Consider the vertex $x_{1,2}$ of $T$.
This vertex partitions the set of edges of $S(G)$ into several sets according to their
distance from $x_{1,2}$.
Obviously, $x_{1,2}$ distinguishes edges belonging to different sets.
Thus, we examine pairs of edges, say $e$ and $f$, belonging to a common set.
It suffices to consider the following three cases.

\bigskip
\noindent
{\bf Case 1}: $d(x_{1,2},e)=d(x_{1,2},f)=0$. Hence, $e=v_1x_{1,2}$ and $f=v_2x_{1,2}$.
We have $d(x_{2,3},e)=2$ and $d(x_{2,3},f)=1$, and so $x_{2,3}$ distinguishes $e,f$.

\bigskip
\noindent
{\bf Case 2}: $d(x_{1,2},e)=d(x_{1,2},f)=1$. If $e=v_1x_{1,j}$ and $f=v_2x_{2,k}$, then
$d(x_{2,3},f)=1$ and $d(x_{2,3},e)\ge 2$, and so, again $x_{2,3}$ distinguishes $e,f$.
Assume that $e=v_2x_{2,j}$ and that $f=v_2x_{2,k}$, where $j<k$.
The case when both $e$ and $f$ are incident to $v_1$ is analogous.
We may assume that $j,k\ne 3$, otherwise $x_{2,3}$ distinguishes $e$ and
$f$. If $k=n$, then there is a vertex $z\in T$ which is at distance at most $1$
from $v_j$ and this vertex distinguishes $e$ and $f$.
If $v_k\in T$, then $v_k$ distinguishes $e$ and $f$.
Thus, it remains to consider the case when $v_j$ and $v_k$ are at distance $1$
from vertices of $T'$.
But then there is a vertex in $T'$ which is at distance $1$ from one of
$v_j,v_k$ and at distance at least $3$ from the other.
And this vertex distinguishes $e$ and $f$.

\bigskip
\noindent
{\bf Case 3}: $d(x_{1,2},e)=d(x_{1,2},f)=2$.
This case can be solved analogously as the previous one, so we omit the
discussion.

\bigskip

Using analogous arguments for the remaining vertices of $T'$, we see that the only
edges of $S(G)$ which may not be distinguished by $T$ are at distance
at least $3$ from every vertex of $T'$.
However, there is not such an edge if $n\equiv 1\pmod 3$.
If $n\equiv 2\pmod 3$, then the edges are $v_{n-1}x_{n-1,n}$ and $v_nx_{n-1,n}$
(if $v_{n-1}v_n\in E(G)$), and since $v_{n-1}\in T$, these edges are
distinguished by $T$.
Finally, if $n\equiv 0\pmod 3$, then there are at most three edges connecting
pairs of vertices of $\{v_{n-2},v_{n-1},v_n\}$ in $G$.
If all the three edges $v_{n-2}v_{n-1}$, $v_{n-2}v_n$ and $v_{n-1}v_n$ are in
$G$, then it remains to consider six edges of the
6-cycle $C=v_{n-2}x_{n-2,n-1}v_{n-1}x_{n-1,n}v_n x_{n-2,n}$ in $S(G)$.
Since $v_{n-2},v_{n-1}\in T$ and these two vertices form an edge metric basis for
$C$, the set $T$ is an edge metric generator for $S(G)$.
The other cases (when $0$, $1$ or $2$ edges of
$v_{n-2}v_{n-1},v_{n-2}v_n,v_{n-1}v_n$ are in $G$) are even simpler.
\end{proof}

The bounds in Theorems~{\ref{thm:bound}} and~{\ref{thm:ebound}} can be violated
if $G$ does not contain a required number of vertex-disjoint paths.
For instance, the complete bipartite graph $K_{n-1,1}$ has $n$ vertices and
there are not two vertex-disjoint paths of length $2$ in $K_{n-1,1}$.
It is easy to see that
$\dim(S(K_{n-1,1}))=\edim(S(K_{n-1,1}))=n-2>\lceil\frac{2n}3\rceil$ if $n\ge 9$.


\subsection{Subdivisions of complete graphs minus a matching}

In the next we show that there are graphs of order $n$ for which the bounds in
Theorems~{\ref{thm:bound}} and~{\ref{thm:ebound}} are sharp.
If $n\not\equiv 0\pmod 3$, then these graphs have edge metric dimension smaller than
the metric dimension.

Let $K^k_n$ be a graph obtained from the complete graph $K_n$ on $n$ vertices
by deleting $k$ independent edges. We have the following statement.

\begin{theorem}
\label{thm:SG}
Let $k\ge 0$ and let $n\ge\max\{4, 3k+4\ell-2\}$, where $\ell=n\!\!\mod 3$.
Then $\edim(S(K^k_n))=\left\lceil\frac{2n-2}3\right\rceil$ and
$\dim(S(K^k_n))=\left\lceil\frac{2n}3\right\rceil$.
\end{theorem}

\begin{proof}
In order to simplify the notation, let us denote $S(K^k_n)$ by $G$.
We begin with $\edim(G)$, so let $S$ be an edge metric generator for $G$.
We construct an auxiliary graph $H$ containing some of the vertices of degree
at least $n-2$ in $G$.
If $v\in S$ and the degree of $v$, $\deg_G(v)$, in $G$ is at least $n-2$,
then we put the vertex $v$ to $H$.
If $x\in S$ and $\deg_G(x)=2$, then we put to $H$ the two neighbors of $x$,
together with the edge joining them.
Hence, if $S$ contains $s_1$ vertices of degree at least $n-2$ and $s_2$
vertices of degree $2$, then $H$ contains exactly $s_2$ edges and at most
$s_1+2s_2$ vertices, since some vertices of degree at least $n-2$ in $S$
can be adjacent to vertices of degree $2$ in $S$ and some pairs of vertices
of degree $2$ in $S$ can have distance $2$ in $G$.
We prove two claims about $H$.

\bigskip
\noindent
{\bf Claim~1}.
{\it $H$ contains at least $n-1$ vertices.}

\medskip
\noindent
Suppose that there are $2$ vertices of degree at least
$n-2$ in $G$, which are not in $H$.
Denote these vertices by $u$ and $v$.
By the assumption on $n$ in the statement, there is $w\in V(K_n^k)$ such that
$uw,vw\in E(K_n^k)$.
Let $y$ (resp. $z$) be the vertex of degree $2$ adjacent to both $w$ and $u$
(resp. $v$). Denote $e=yw$ and $f=zw$.

If $s\in S$ and $\deg_G(s)=2$, then either $s$ is adjacent to $w$, in which case
$d_G(s,e)=d_G(s,f)=1$; or $s$ is not adjacent to $w$, in which case
$d_G(s,e)=d_G(s,f)=3$, since one of the two neighbors of $s$ is connected to $w$
by an edge in $K_n^k$ and $s$ is adjacent to neither of $u$ and $v$.
On the other hand, if $s\in S$ and $\deg_G(s)\ge n-2$, then either
$sw\in E(K_n^k)$, in which case $d_G(s,e)=d_G(z,f)=2$; or $sw\notin  E(K_n^k)$,
in which case $d_G(s,e)=d_G(z,f)=3$, since $su,sv\in E(K_n^k)$.
Consequently, $e,f$ are not distinguished by $S$, a contradiction.
This proves the claim. 
\bigskip

\noindent
{\bf Claim~2}.
{\it If $uv\in E(H)$, then either $S$ contains one of $u$ and $v$, or
there is another edge adjacent to $uv$ in $H$.}

\medskip
\noindent
Suppose that $\deg_G(q)=2$ and $q\in S$.
Let $u$ and $v$ be the vertices of degree at least $n-2$ adjacent to $q$.
Moreover, suppose that $S$ does not contain $u,v$ and $S$ does not contain
a vertex of degree $2$ adjacent to $u$ or $v$ other than $q$.
By the assumption on $n$ in the statement, there is $w\in V(K_n^k)$ such that
$uw,vw\in E(K_n^k)$.
Let $y$ (resp. $z$) be the vertex of degree $2$ adjacent to both $w$ and $u$
(resp. $v$). Denote $e=yw$ and $f=zw$. We proceed analogously as above.

If $s\in S\setminus\{q\}$ and $\deg_G(s)=2$, then either $s$ is adjacent to $w$,
in which case $d_G(s,e)=d_G(s,f)=1$; or $s$ is not adjacent to $w$, in which case
$d_G(s,e)=d_G(s,f)=3$, since one of the two neighbors of $s$ is connected to $w$
by an edge in $K_n^k$, and $s$ is neither adjacent to $u$ nor to $v$.
On the other hand, if $s\in S$ and $\deg_G(s)\ge n-2$, then either
$sw\in E(K_n^k)$, in which case $d_G(s,e)=d_G(s,f)=2$; or $sw\notin  E(K_n^k)$,
in which case $d_G(s,e)=d_G(s,f)=3$, since $su,sv\in E(K_n^k)$.
Since $d_G(q,e)=d_G(q,f)=2$, we obtain that $e,f$ are not identified by $S$,
which is a contradiction. This establishes the claim. 
\bigskip

Now, let $C$ be a connected component of $H$.
If $C$ has only one vertex, then $C$ was created by using one vertex of degree at
least $n-2$ in $S$.
If $C$ has two vertices, then it was created by using one vertex of degree $2$
in $S$, and at least one vertex of degree at least $n-2$ in $S$, by Claim~2.
If $C$ has $t$ vertices, $t\ge 3$, then $C$ was created using at least $t-1$
vertices of degree $2$ in $S$ and maybe some vertices of degree at least
$n-2$ in $S$.
By Claim~1, $H$ contains at least $n-1$ vertices, and so $S$ contains
at least $2\lfloor\frac{n-1}3\rfloor+c$ vertices.
Here $c=0$ if $n\equiv 1\pmod 3$, $c=1$ if $n\equiv 2\pmod 3$ and
$c=2$ if $n\equiv 0\pmod 3$.
Hence $2\lfloor\frac{n-1}3\rfloor+c=\lceil\frac{2n-2}3\rceil$, and so
$|S|\ge\lceil\frac{2n-2}3\rceil$.

To prove that $\edim(G)\le\lceil\frac{2n-2}3\rceil$ it suffices to show
that there are $\lfloor\frac{n-1}3\rfloor$ vertex disjoint paths
of length $2$ in $G$, by Theorem~{\ref{thm:ebound}}.
But this is a consequence of the fact that $K^k_n$ has
a Hamiltonian cycle (by using the Dirac theorem).

\medskip

We next consider $\dim(G)$, so let $S$ be a metric generator for $G$.
Let $H$ be an auxiliary graph constructed as above.
We prove three claims about $H$ as before, where the first two are similar as before.

\bigskip
\noindent
{\bf Claim~3}. {\it $H$ contains at least $n-1$ vertices.}

\medskip
\noindent
Suppose that there are $2$ vertices $u$ and $v$ of degree at least $n-2$ in
$G$, which are not in $H$.
By the assumption on $n$ in the statement, there is $w\in V(K^k_n)$ such that
$uw,vw\in E(K^k_n)$.
Let $y$ (resp. $z$) be the vertex of degree $2$ adjacent to both $w$ and $u$
(resp. $v$).

If $s\in S$ and $\deg_G(s)=2$, then either $s$ is adjacent to $w$, in which case
$d_G(s,y)=d_G(s,z)=2$; or $s$ is not adjacent to $w$, in which case
$d_G(s,y)=d_G(s,z)=4$, since one of the two neighbors of $s$ is connected to $w$
by an edge in $K^k_n$ and $s$ is adjacent to neither of $u$ and $v$.
On the other hand, if $s\in S$ and $\deg_G(s)\ge n-2$, then either
$sw\in E(K^k_n)$, in which case $d_G(s,y)=d_G(s,z)=3$; or $sw\notin E(K^k_n)$,
in which case $d_G(s,y)=d_G(s,z)=3$, since $su,sv\in E(K^k_n)$.
Consequently, $y,z$ are not identified by $S$, a contradiction.
This proves the claim. 

\bigskip
\noindent
{\bf Claim~4}.
{\it If $uv\in E(H)$, then either $S$ contains one of $u$ and $v$, or
there is another edge adjacent to $uv$ in $H$.}

\medskip
\noindent
Suppose that $\deg_G(q)=2$ and that $q\in S$.
Let $u$ and $v$ be the vertices of degree at least $n-2$ adjacent to $q$.
Moreover, suppose that $S$ does not contain $u,v$ and $S$ does not contain
a vertex of degree $2$ adjacent to $u$ or $v$ other than $q$.
By the assumption on $n$ in the statement, there is $w\in V(K^k_n)$ such that
$uw,vw\in E(K^k_n)$.
Let $y$ (resp. $z$) be the vertex of degree $2$ adjacent to both $w$ and $u$
(resp. $v$). We proceed analogously as above.

If $s\in S\setminus \{q\}$ and $\deg_G(s)=2$, then either $s$ is adjacent to $w$, in which case
$d_G(s,y)=d_G(s,z)=2$; or $s$ is not adjacent to $w$, in which case
$d_G(s,y)=d_G(s,z)=4$, since one of the two neighbors of $s$ is connected to $w$
by an edge in $K^k_n$ and $s$ is neither adjacent to $u$ nor to $v$.
On the other hand, if $s\in S$ and $\deg_G(s)\ge n-2$, then either
$sw\in E(K^k_n)$, in which case $d_G(s,y)=d_G(s,z)=3$; or $sw\notin  E(K^k_n)$,
in which case $d_G(s,y)=d_G(s,z)=3$, since $su,sv\in E(K^k_n)$.
Since $d_G(q,y)=d_G(q,z)=2$, we have that $y,z$ are not distinguished
by $S$, a contradiction.
This proves the claim. 

\bigskip
\noindent
{\bf Claim~5}.
{\it If there is a vertex of $K^k_n$ which is not in $H$, then $H$ can
contain at most $k$ components with exactly two edges whose end vertices are
not in $S$.}

\medskip
\noindent
Suppose that there is a vertex of $K^k_n$, say $r$, which is not in $H$.
Let $C$ be a connected component of $H$ containing exactly $2$ edges
and let $V(C)\cap S=\emptyset$. Denote the edges of $C$ by $uv$ and $vw$.

Suppose first that $uw,vr\in E(K^k_n)$. Denote by $y$ the vertex of degree $2$ adjacent to $u$ and $w$, and denote by
$z$ the vertex of degree $2$ adjacent to $v$ and $r$ in $G$.  For any $s\in S$,  we consider two possibilities.
First, if $\deg_G(s)=2$, then either $s$ is one of the two vertices which caused an edge of $C$, in which case $d_G(s,y)=d_G(s,z)=2$;
or $s$ is not adjacent to a vertex of $\{u,v,w,r\}$, in which case $d_G(s,y)=d_G(s,z)=4$, since one of the two neighbors of $s$ is
connected to $r$ (and one of them is connected to $u$) by an edge in $K^k_n$.
Second,  if $s\in S$ and $\deg_G(s)\ge n-2$, then $sv\in E(K^k_n)$ or
$sr\in E(K^k_n)$, and also $su\in E(K^k_n)$ or $sw\in E(K^k_n)$, which
implies $d_G(s,y)=d_G(s,z)=3$. Thus,  in both possibilities, we conclude that $y,z$ are not identified by $S$,
a contradiction.

This means that at least one of the edges $uw$ and $vr$ must be missing in
$K^k_n$. Observe that, if $H$ has another component $C'$ on two edges, say $u'v'$ and
$v'w'$, with $V(C')\cap S=\emptyset$, then for the pairs $u'w'$ and
$v'r$ we have $\{u'w',v'r\}\cap\{uw,vr\}=\emptyset$.
Since there are only $k$ pairs of vertices which are not connected by an
edge in $K^k_n$, $H$ can have at most $k$ components with exactly
two edges whose end vertices are not in $S$.
This proves Claim 5. 
\bigskip

In consequence, there are two cases to consider.
If $|V(H)|=n-1$, then $|S|\ge \beta=2k+\lceil 3\frac{n-3k-1}4\rceil$.
On the other hand, if $|V(H)|=n$, then $|S|\ge \gamma=\lceil 2\frac n3\rceil$.
Observe that $\gamma\le\beta$ is equivalent with
\begin{equation}
\label{eq:1}
\Big\lceil\frac{2n}3\Big\rceil-1<\Big\lceil\frac{3n-k-3}4\Big\rceil.
\end{equation}
Let $n=3t+\ell$, where $\ell= n\mod 3$.
Then (\ref{eq:1}) is equivalent with
$$
2t+\ell-1<\frac{9t+3\ell-k-3}4,
$$
and consequently with
$$
k+\ell-1<t.
$$
Thus, $\gamma\le\beta$ is equivalent with $3k+4\ell-3<3t+\ell=n$, and so
with $n\ge 3k+4\ell-2$.
By the assumption on $n$ in the statement we get
$|S|\ge\lceil\frac{2n}3\rceil$.

To prove $\dim(S(K^k_n))\le\lceil\frac{2n}3\rceil$ it suffices to show that
there are $\lfloor\frac n3\rfloor$ vertex disjoint paths of length $2$ in
$K^k_n$, by Theorem~{\ref{thm:bound}}.
Analogously as above, this is a consequence of the fact that $K^k_n$ has
a Hamiltonian cycle.
\end{proof}

Since $3k+4\ell-2\le 3k+6$, we have the following corollary of
Theorem~{\ref{thm:SG}}.

\begin{cor}
\label{cor:SG}
Let $k\ge 1$ and let $n\ge 3k+6$.
Then $\dim(S(K^k_n))=\edim(S(K^k_n))$ if $n\equiv 0\pmod 3$ and
$\dim(S(K^k_n))=\edim(S(K^k_n))+1$ otherwise.
\end{cor}

And for $S(K_n)$ we have the following statement.

\begin{proposition}
\label{prop:SK}
Let $n\ge 4$. Then $\edim(S(K_n))=\lceil\frac{2n-2}3\rceil$ and
$\dim(S(K_n))=\lceil\frac{2n}3\rceil$, with the unique exception when
$\dim(S(K_5))=3$ instead of $4$.
\end{proposition}

\begin{proof}
For $n\ne5$ the statement is a direct consequence of Theorem~{\ref{thm:SG}}.
For $n=5$ the only problem is that $4=\gamma>\beta=3$ in this case (see the
proof of Theorem~{\ref{thm:SG}}).
Using the notation of Theorem~{\ref{thm:bound}}, let
$T=\{x_{1,2},x_{1,3},x_{1,4}\}$.
Then the corresponding graph $H$ has $4$ vertices and it is easy to check
that $T$ is a metric generator for $S(K_5)$.
\end{proof}

By Proposition~{\ref{prop:SK}}, we have $\edim(S(K_{3t+1}))=2t$ and
$\dim(S(K_{3t+1}))=2t+1$ for every $t\ge 1$, and also $\edim(S(K_{3t+2}))=2t+1$ and $\dim(S(K_{3t+2}))=2t+2$ for every $t\ge 2$.
Hence, we have the following statement.

\begin{lemma}
\label{lem:difference1}
For every $c_1\ge 2$, $c_1\ne 3$, there is an integer $q$ such that
$\edim(S(K_q))=c_1$ and $\dim(S(K_q))=c_1+1$.
\end{lemma}

\section{Proof of Theorem \ref{thm:main}}

Clearly, Lemma~{\ref{lem:difference1}} shows that
Theorem~{\ref{thm:main}} is true for the case when $\dim(G)-\edim(G)=1$.
In the following we consider the case with a larger difference.

Let $c_1\ge 4$ and $c_2\ge c_1+2$ be two integers. We are aimed to construct a graph $G_{c_1,c_2}$ for which $\edim(G_{c_1,c_2})=c_1$ and $\dim(G_{c_1,c_2})=c_2$. Let $k=c_2-c_1$ and observe that $k\ge 2$.
We first take $k$ graphs $G_1,G_2,\dots,G_k$ and connect them into a structure
that resembles a path. The graphs $G_1,G_2,\dots,G_{k-1}$ are isomorphic to $S(K_7)$, while $G_k$ is
isomorphic to a graph $S(K_q)$, for some integer $q$, satisfying $\edim(S(K_q))=c_1$.
Since $c_1\ge 4$, by Lemma~{\ref{lem:difference1}}, there always exists the required integer $q$.

Now we denote the vertices of degree at least $3$ in $G_i$ by
$v^i_1,v^i_2,\dots v^i_{|V(G_i)|}$, and we denote the vertex of degree $2$ adjacent to
$v^i_a$ and $v^i_b$ by $x^i_{a,b}$.
By using this notation, we add to the disjoint union of $G_1,G_2,\dots,G_k$ the
edges $x^i_{1,2}x^{i+1}_{2,3}$ and $x^i_{4,5}x^{i+1}_{5,6}$, $1\le i\le
k-1$, see Figure \ref{fig:G_k} for a sketch of the construction.

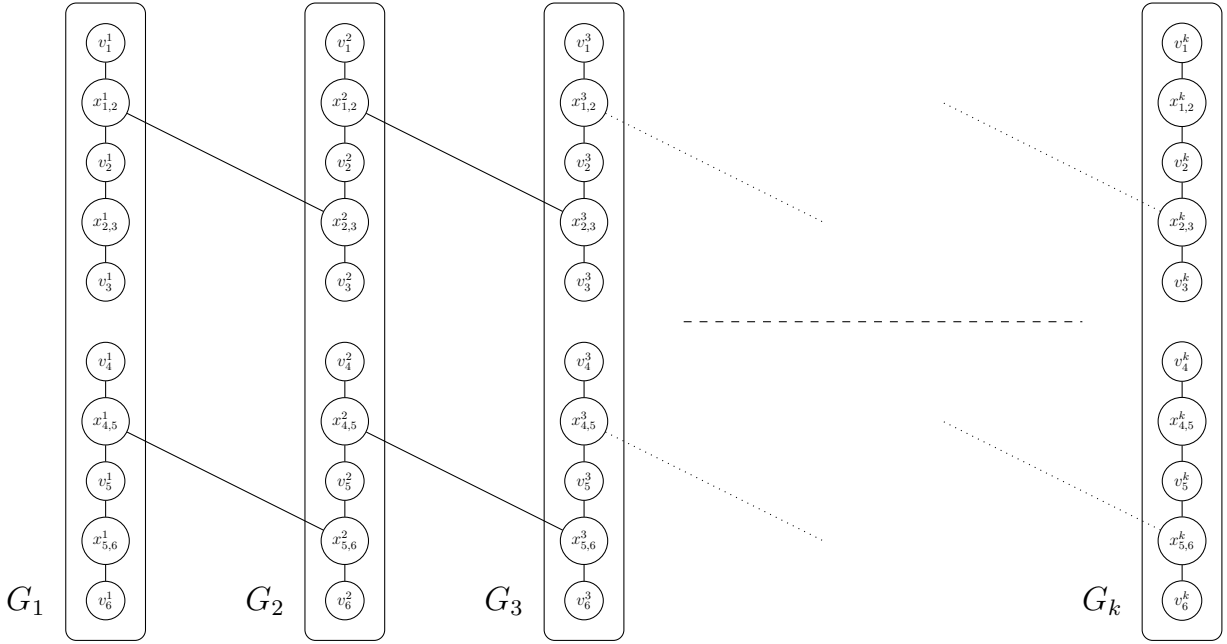
\begin{figure}
  \centering
\begin{tikzpicture}[scale=.53, transform shape]
\node [draw, shape=circle] (a1) at  (1,1) {$v_3^1$};
\node [draw, shape=circle] (a2) at  (1,4) {$v_2^1$};
\node [draw, shape=circle] (a3) at  (1,7) {$v_1^1$};
\node [draw, shape=circle] (b1) at  (1,-1) {$v_4^1$};
\node [draw, shape=circle] (b2) at  (1,-4) {$v_5^1$};
\node [draw, shape=circle] (b3) at  (1,-7) {$v_6^1$};

\node [draw, shape=circle] (s1) at  (1,2.5) {$x_{2,3}^1$};
\node [draw, shape=circle] (s2) at  (1,5.5) {$x_{1,2}^1$};
\node [draw, shape=circle] (s3) at  (1,-2.5) {$x_{4,5}^1$};
\node [draw, shape=circle] (s4) at  (1,-5.5) {$x_{5,6}^1$};

\draw[rounded corners] (0,-8) rectangle (2,8);
\node [scale=2] at (-1,-7) {$G_1$};

\draw(a3)--(s2)--(a2)--(s1)--(a1);
\draw(b3)--(s4)--(b2)--(s3)--(b1);

\node [draw, shape=circle] (a11) at  (7,1) {$v_3^2$};
\node [draw, shape=circle] (a21) at  (7,4) {$v_2^2$};
\node [draw, shape=circle] (a31) at  (7,7) {$v_1^2$};
\node [draw, shape=circle] (b11) at  (7,-1) {$v_4^2$};
\node [draw, shape=circle] (b21) at  (7,-4) {$v_5^2$};
\node [draw, shape=circle] (b31) at  (7,-7) {$v_6^2$};

\node [draw, shape=circle] (s11) at  (7,2.5) {$x_{2,3}^2$};
\node [draw, shape=circle] (s21) at  (7,5.5) {$x_{1,2}^2$};
\node [draw, shape=circle] (s31) at  (7,-2.5) {$x_{4,5}^2$};
\node [draw, shape=circle] (s41) at  (7,-5.5) {$x_{5,6}^2$};

\draw[rounded corners] (6,-8) rectangle (8,8);
\node [scale=2] at (5,-7) {$G_2$};

\draw(a31)--(s21)--(a21)--(s11)--(a11);
\draw(b31)--(s41)--(b21)--(s31)--(b11);

\node [draw, shape=circle] (a111) at  (13,1) {$v_3^3$};
\node [draw, shape=circle] (a211) at  (13,4) {$v_2^3$};
\node [draw, shape=circle] (a311) at  (13,7) {$v_1^3$};
\node [draw, shape=circle] (b111) at  (13,-1) {$v_4^3$};
\node [draw, shape=circle] (b211) at  (13,-4) {$v_5^3$};
\node [draw, shape=circle] (b311) at  (13,-7) {$v_6^3$};

\node [draw, shape=circle] (s111) at  (13,2.5) {$x_{2,3}^3$};
\node [draw, shape=circle] (s211) at  (13,5.5) {$x_{1,2}^3$};
\node [draw, shape=circle] (s311) at  (13,-2.5) {$x_{4,5}^3$};
\node [draw, shape=circle] (s411) at  (13,-5.5) {$x_{5,6}^3$};

\draw[rounded corners] (12,-8) rectangle (14,8);
\node [scale=2] at (11,-7) {$G_3$};

\draw(a311)--(s211)--(a211)--(s111)--(a111);
\draw(b311)--(s411)--(b211)--(s311)--(b111);

\draw(s11)--(s2);
\draw(s41)--(s3);
\draw(s111)--(s21);
\draw(s411)--(s31);
\draw[dotted](s211)--(19,2.5);
\draw[dotted](s311)--(19,-5.5);

\draw[dashed](15.5,0)--(25.5,0);

\node [draw, shape=circle] (a12) at  (28,1) {$v_3^k$};
\node [draw, shape=circle] (a22) at  (28,4) {$v_2^k$};
\node [draw, shape=circle] (a32) at  (28,7) {$v_1^k$};
\node [draw, shape=circle] (b12) at  (28,-1) {$v_4^k$};
\node [draw, shape=circle] (b22) at  (28,-4) {$v_5^k$};
\node [draw, shape=circle] (b32) at  (28,-7) {$v_6^k$};

\node [draw, shape=circle] (s12) at  (28,2.5) {$x_{2,3}^k$};
\node [draw, shape=circle] (s22) at  (28,5.5) {$x_{1,2}^k$};
\node [draw, shape=circle] (s32) at  (28,-2.5) {$x_{4,5}^k$};
\node [draw, shape=circle] (s42) at  (28,-5.5) {$x_{5,6}^k$};

\draw[rounded corners] (27,-8) rectangle (29,8);
\node [scale=2] at (26,-7) {$G_k$};

\draw(a32)--(s22)--(a22)--(s12)--(a12);
\draw(b32)--(s42)--(b22)--(s32)--(b12);

\draw[dotted](s12)--(22,5.5);
\draw[dotted](s42)--(22,-2.5);
\end{tikzpicture}
\caption{A sketch of a graph $G_{c_1,c_2}$ for some $c_1,c_2$ as described above. Note that only those vertices and edges which have influence on the construction have been drawn. Also, recall that $G_k$ might be different from the remaining $G_i$.}
\label{fig:G_k}
\end{figure}

In the proof of $\edim(G_{c_1,c_2})=c_1$ and $\dim(G_{c_1,c_2})=c_2$ we use
the following lemma.

\begin{lemma}
\label{lem:bound_sequence}
Let $G$ be a graph with a subgraph $H$.
Let $z_1,z_2,\dots,z_t$ be vertices of $H$ that have neighbors in
$V(G)\setminus V(H)$. If $\dim(H)=c$ $($resp. $\edim(H)=c$$)$,
then every $($resp. edge$)$ metric generator for $G$
contains at least $c-t$ vertices of $V(H)\setminus\{z_1,z_2,\dots,z_t\}$.
\end{lemma}

\begin{proof}
We present the proof for $\dim(G)$ only, since for $\edim(G)$ it is
almost the same.
Consider the equivalence classes of $V(H)$ formed by distances from
$z_1,z_2,\dots,z_t$.
Two vertices, say $u_1$ and $u_2$, are in the same equivalence class if
$d(u_1,z_i)=d(u_2,z_i)$ for every $i$, $1\le i\le t$.

Now take two vertices, say $y_1$ and $y_2$, from the same equivalence class,
and take a vertex $z\in V(G)\setminus V(H)$.
Denote $a_i=d(z,z_i)$ and $b_i=d(z_i,y_1)$ ($=d(z_i,y_2)$), $1\le i\le t$.
Then $d(z,y_1)=\min\{a_i+b_i;1\le i\le t\}=d(z,y_2)$.
Hence, no  vertex outside $H$ can separate $y_1$ and $y_2$ by means of
distances.
Thus, even if all the vertices of $(V(G)\setminus
V(H))\cup\{z_1,z_2,\dots,z_t\}$ are in a metric generator, they will not
distinguish vertices of $H$ which are not distinguished by the set
$\{z_1,z_2,\dots,z_t\}$ alone.
Consequently, since $\dim(H)=c$, to distinguish vertices of $H$
every metric generator for $G$ must have at least $c-t$ vertices of
$V(H)\setminus\{z_1,z_2,\dots,z_t\}$.
\end{proof}

Now we determine the metric and edge metric dimensions of $G_{c_1,c_2}$,
which shall complete the proof of Theorem~{\ref{thm:main}}.

\begin{lemma}
\label{lem:c12}
For every pair of integers $c_1$ and $c_2$, such that $c_1\ge 4$ and
$c_2\ge c_1+2$, we have $\edim(G_{c_1,c_2})=c_1$ and
$\dim(G_{c_1,c_2})=c_2$.
\end{lemma}

\begin{proof}
We start with the lower bound.
Since $\dim(S(K_7))=5$ and $\dim(S(K_q))=c_1+1$, every metric generator for
$G_{c_1,c_2}$ must contain at least $3$ vertices of
$V(G_1)\setminus\{x^1_{1,2},x^1_{4,5}\}$, at least $1$ vertex of
$V(G_i)\setminus\{x^i_{1,2},x^i_{2,3},x^i_{4,5},x^i_{5,6}\}$, $2\le i\le k-1$, and at least
$c_1-1$ vertices of $V(G_k)\setminus\{x^k_{2,3},x^k_{5,6}\}$, by
Lemma~{\ref{lem:bound_sequence}}.
Thus, every metric generator for $G_{c_1,c_2}$ must have at least
$3+(k{-}2)+(c_1{-}1)=c_2$ vertices, and so $\dim(G_{c_1,c_2})\ge c_2$.

Analogously, since $\edim(S(K_7))=4$ and $\edim(S(K_q))=c_1$, every edge
metric generator for $G_{c_1,c_2}$ must contain at least $2$ vertices of
$V(G_1)\setminus\{x^1_{1,2},x^1_{4,5}\}$ and at least $c_1-2$ vertices of
$V(G_k)\setminus\{x^k_{2,3},x^k_{5,6}\}$, by
Lemma~{\ref{lem:bound_sequence}}.
Thus, $\edim(G_{c_1,c_2})\ge c_1$.

For the upper bound we use the bases defined in the proofs of
Theorems~{\ref{thm:bound}} and ~{\ref{thm:ebound}}.
Observe that $G_i\cong S(K_r)$ where $r\ge 7$, so that all these bases
contain $x^i_{1,2},x^i_{2,3},x^i_{4,5},x^i_{5,6}$ in $G_i$.
We start with $\edim(G_{c_1,c_2})$.
Thus, we show that $T=\{x^1_{2,3},x^1_{5,6},x^k_{1,2},x^k_{4,5},\dots\}$
is an edge metric basis in $G_{c_1,c_2}$ (observe that $T$ is composed by
vertices only from $G^1$ and $G^k$).
For this we show that the partition of edges of $G_1$ via distances from
$x^1_{1,2}$ can be modelled by distances from $x^k_{1,2}$ and $x^k_{4,5}$.
And, by symmetry we can have a similar statement for $x^1_{4,5}$.
Indeed, let $d(x^k_{1,2},x^1_{1,2})=a$. In fact, $a=3(k-1)$.
Then $d(x^k_{4,5},x^1_{4,5})=a$ as well and
$d(x^k_{1,2},x^1_{4,5})=d(x^k_{4,5},x^1_{1,2})=a+2$.
Thus, edges of $G_1$ at distance $0$ from $x^1_{1,2}$ are those which are in
$G_1$ (see below) and at distance $a$ from $x^k_{1,2}$.
Edges of $G_1$ at distance $1$ from $x^1_{1,2}$ are those which are in
$G_1$ and at distance $a+1$ from $x^k_{1,2}$.
Finally (it suffices to consider edges at distance $0$, $1$ and $2$ from $x^1_{1,2}$; see the proof of Theorem~{\ref{thm:ebound}}), edges of $G_1$
at distance $2$ from $x^1_{1,2}$ are those which are in
$G_1$, at distance $a+2$ from $x^k_{1,2}$ and at distance greater than $a$
from $x^k_{4,5}$ (to avoid the two edges at distance $0$ from $x^1_{4,5}$).
Analogously, we can reconstruct the partition of $E(G_i)$ from any one of the
vertices of $G_i$ which have a neighbour outside $G_i$, $1\le i\le k$.

Hence, it remains to show how the distances from $x^1_{2,3}$, $x^1_{5,6}$,
$x^k_{1,2}$ and $x^k_{4,5}$ can be used to detect edges of $G_i$ and how to
identify the edges connecting $G_i$ with $G_{i+1}$, $1\le i\le k-1$.
As regards the edges connecting $G_i$ with $G_{i+1}$, observe that
$x^i_{1,2}x^{i+1}_{2,3}$ (resp. $x^i_{4,5}x^{i+1}_{5,6}$) is the unique edge at
distance $2+3(i{-}1)$ from $x^1_{2,3}$ (resp. $x^1_{5,6}$) and at distance
$2+3(k{-}i{-}1)$ from $x^k_{1,2}$ (resp. $x^k_{4,5}$).
Finally, edges of $G_1$ are those which are at distance at most $3$ from
both $x^1_{2,3}$ and $x^1_{5,6}$;
edges of $G_2$ are those which are not in
$E(G_1)\cup\{x^1_{1,2}x^2_{2,3},x^1_{4,5}x^2_{5,6}\}$ and at distance at
most $6$ from both $x^1_{2,3}$ and $x^1_{5,6}$; etc.
Hence, $T$ distinguishes all edges of $G_{c_1,c_2}$.
Since $|T|=c_1$, we have $\edim(G_{c_1,c_2})\le c_1$.

The proof of $\dim(G_{c_1,c_2})\le c_2$ is analogous.
\end{proof}

%
%
%
%

\vskip 1pc
\noindent{\bf Acknowledgements.}~~The first author (M. Knor) acknowledges the partial support by Slovak research grants VEGA 1/0142/17, VEGA 1/0238/19, APVV--15--0220, APVV--17--0428. The second author (R. \v Skrekovski) acknowledges the Slovenian research agency ARRS, program no.\ P1--0383 and project no. J1-1692. The last author (Ismael G. Yero) has been partially supported by the Spanish Ministry of Science and Innovation through the grant PID2019-105824GB-I00.

\end{document}